\documentclass[12pt]{amsart}
\usepackage{amssymb,enumerate,ifthen}
\usepackage[mathscr]{eucal}
\usepackage[utf8]{inputenc}
\usepackage[T1]{fontenc}

\makeatletter
\@namedef{subjclassname@2010}{%
  \textup{2020} Mathematics Subject Classification}
\makeatother

\DeclareUnicodeCharacter{1EF3}{\`y}

\newtheorem{thm}{Theorem}
\newtheorem{prop}{Proposition}[section]
\newtheorem{lem}[prop]{Lemma}
\newtheorem{cor}[prop]{Corollary}
\newcommand{\sF}{\mathscr{F}}
\newcommand{\sG}{\mathscr{G}}
\newcommand{\CM}{\mathcal{CM}}
\newcommand{\G}{\mathscr{G}}

\theoremstyle{definition}

\usepackage{cite}
\newcommand\R{\mathbb{R}}

\renewcommand\P{\mathscr{P}}
\newcommand\N{\mathbb{N}}

\newcommand\MM{\mathbf{M}}
\newcommand\KK{\mathbf{K}}

\newcommand{\abs}[1]{\left| #1 \right| }

\newcommand\dd{\mathbf{d}}
\newcommand\UU{U}
\newcommand\LL{L}
\newcommand\UUG{\mathcal{U}}
\newcommand\LLG{\mathcal{L}}
\newcommand\QA[1]{A^{[#1]}}

\author[P. Pasteczka]{Pawe\l{} Pasteczka}
\address{Institute of Mathematics \\ University of the National Education Commission, \\ Podchor\k{a}\.zych str. 2, 30-084 Krak\'ow, Poland}
\email{pawel.pasteczka@up.krakow.pl}

\subjclass[2010]{26E60, 39B12, 39B22}
\keywords{Invariant means, extended means, multivariable generalizations, envelopes}

\DeclareMathOperator{\sign}{sign}

\newcommand{\restr}{\!\!\upharpoonright}
\numberwithin{equation}{section}
\def\eq#1{{\rm(\ref{#1})}}
\def\Eq#1#2{\ifthenelse{\equal{#1}{*}}
  {\begin{equation*}\begin{aligned}[]#2\end{aligned}\end{equation*}}
  {\begin{equation}\begin{aligned}[]\label{#1}#2\end{aligned}\end{equation}}}
\title[Multivariable generalizations of bivariate means...]{Multivariable generalizations of bivariate means via invariance}
\begin{document}

\begin{abstract}
For a given $p$-variable mean $M \colon I^p \to I$ ($I$ is a subinterval of $\mathbb{R}$), following (Horwitz,~2002) and (Lawson and Lim,~2008),
 we can define (under certain assumption) its $(p+1)$-variable $\beta$-invariant extension as the unique solution $K \colon I^{p+1} \to I$ of the functional equation
\begin{align*}
K\big(M(x_2,\dots,x_{p+1})&,M(x_1,x_3,\dots,x_{p+1}),\dots,M(x_1,\dots,x_p)\big)\\
&=K(x_1,\dots,x_{p+1}), \text{ for all }x_1,\dots,x_{p+1} \in I
\end{align*}
in the family of means.

Applying this procedure iteratively we can obtain a mean which is defined for vectors of arbitrary lengths starting from the bivariate one. The aim of this paper is to study the properties of such extensions. \end{abstract}
\maketitle
\section{Introduction}
The problem of the multi-variable generalization of bivariate means is very natural. Regrettably, such extensions remain unknown for many families of means. For example, in the family of generalized logarithmic means ${\mathcal E}_s \colon \R_+^2 \to \R_+$ ($s \in \R\setminus\{-1,0\}$) defined as 
\Eq{*}{
{\mathcal E}_s(x,y):=\Big(\frac{x^{s+1}-y^{s+1}}{(s+1)(x-y)}\Big)^{1/s},
}
Stolarsky means, Heronian mean, several means related to Pythagorean, etc. (see  for example \cite{Bul03} for their definition). 

The purpose of this note is to provide a broad approach to this problem. Namely, we extend some ideas of Aumann  \cite{Aum34,Aum44}, Horwitz \cite{Hor02} and Lawson--Lim \cite{LawLim08} to generalize bivariate means to the multivariable setting. More precisely, for a given $k$-variable symmetric, continuous, and strict mean on an interval, we can apply the so-called barycentric operator to generate the $(k+1)$-variable mean on the same interval. Then we use this procedure iteratively in order to get the desired extension.

At the very beginning, let us introduce the family of Gini means, which will be very helpful to illustrate the problem. 
Namely, in 1938, C.~Gini~\cite{Gin38} introduced the generalization of power means. For $r,s\in\R$, the \emph{Gini mean} $\G_{r,s}$ of 
positive variables $x_1,\dots,x_n$ ($n \in \N$) is defined as follows:
\Eq{*}{
  \G_{r,s}(x_1,\dots,x_n)
   :=\left\{\begin{array}{ll}
    \left(\dfrac{x_1^r+\cdots+x_n^r}
           {x_1^s+\cdots+x_n^s}\right)^{\frac{1}{r-s}} 
      &\mbox{if }r\neq s, \\[4mm]
     \exp\left(\dfrac{x_1^r\ln(x_1)+\cdots+x_n^r\ln(x_n)}
           {x_1^r+\cdots+x_n^r}\right) \quad
      &\mbox{if }r=s.
    \end{array}\right.
}
Clearly, in the particular case $s=0$, the mean $\G_{r,0}$ becomes the $r$th Power mean $\P_r$. It is also obvious that $\G_{s,r}=\G_{r,s}$. 

It can be easily shown that for all $r \in \R$ and $x,y \in \R_+$ we have $\G_{r,-r}(x,y)=\sqrt{xy}$. This equality, however, fails to be valid for more than two arguments. Thus (a~priori) it could happen that two different means coincide in the bivariate setting. As a consequence, we cannot recover the multi-variable mean based only on its two-variable restriction. 

On the other hand, if $f\colon I \to \R$ ($I$ is an interval) is a continuous, strictly monotone function and  $M=\QA{f}\colon \bigcup_{n=1}^\infty I^n \to I$ is a quasiarithmetic mean (see section~\ref{sec:QA} for definition) then it solves the functional equation
\Eq{E:152}{
M(x_1,\dots,x_{k+1})&=M\big(M(x_2,\dots,x_{k+1}),M(x_1,x_3,\dots,x_{k+1}),\dots,\\
& M(x_1,\dots,x_k)\big) \text{ for all }k \ge 2\text{ and }x_1\dots,x_{k+1} \in I.
}
Moreover, one can show that $M=\QA{f}$ is the only mean which solves this equation and such that $M(x_1,x_2)=\QA{f}(x_1,x_2)$ for all $x_1,x_2 \in I$. 
As a consequence, we can utilize \eq{E:152} to calculate the value of quasiarithmetic mean for a vector of an arbitrary length based only on its bivariate restriction. The aim of this paper is to generalize the procedure above, to extend other bivariate means to vectors of an arbitrary length.

\subsection{General framework}

Formally, a \emph{mean of order $k$}, or \emph{$k$-mean} for short, on
a set $X$ is a function $\mu \colon X^k \to X$ satisfying $\mu(x, \dots , x) = x$ for all $x \in X$.
In the twentieth century, the theory of topological means, that is, symmetric means on topological spaces for which the mean operation is continuous, was of great interest. This work was pioneered by Aumann~\cite{Aum34}, who showed, among other things, that no sphere admits such a mean \cite{Aum44}.

Now we proceed to the notion of $\beta$-invariant extension introduced by Horwitz \cite{Hor02}. Given a set $X$ and a $k$-mean $\mu \colon X^k\to X$, the \emph{barycentric operator} $\beta=\beta_\mu \colon X^{k+1}\to X^{k+1}$ is defined by
\Eq{*}{
\beta_\mu(x):=(\mu(x^{\vee1}),\dots,\mu(x^{\vee(k+1)})),
}
where $x^{\vee j}$ is a vector which is obtain by removing the $j$-th coordinate in the vector $x$, that is $x^{\vee j}=(x_i)_{i\ne j}$.
For a topological $k$-mean, we say that the barycentric map $\beta$ is \emph{power convergent} if for each $x \in X^{k+1}$, we have $\lim _{n\to \infty} \beta^n(x)=(x^*,\dots,x^*)$ for some $x^*\in X$.

A mean $\nu \colon X^{k+1} \to X$ is a $\beta$-invariant extension of $\mu \colon X^k \to X$ if $\nu \circ \beta_\mu=\nu$, that is 
\Eq{*}
{
\nu(\mu(x^{\vee1}),\dots,\mu(x^{\vee(p+1)}))=\nu(x), \quad x \in X^{k+1}.
}
Let us now recall an important result by Lawson--Lim \cite{LawLim08}.

\begin{prop}[\!\!\cite{LawLim08}, Proposition 2.4]
Assume that $\mu \colon X^k \to X$ is a topological k-mean and that the corresponding barycentric operator $\beta_\mu$ is power convergent. Define
$\widetilde \mu : X^{k+1} \to X$ by $\widetilde \mu(x) = x^*$, where $\lim_{n\to\infty} \beta_\mu^n(x)=(x^*,\dots,x^*)$.
\begin{enumerate}[(i)]
    \item $\tilde \mu : X^{k+1} \to X$ is a $(k + 1)$-mean X that is a $\beta$-invariant extension of $\mu$.
    \item Any continuous mean on $X^{k+1}$ that is a $\beta$-invariant extension of $\mu$ must equal $\tilde \mu$.
    \item If $\mu$ is symmetric, so is $\tilde \mu$.
\end{enumerate}
\end{prop}

\subsection{Properties of means on the interval} For the sake of completeness, let us introduce formally $\N:=\{1,\dots\}$, and $\N_p:=\{1,\dots,p\}$ (where $p\in \N$).

Throughout this note, $I$ is a subinterval of $\R$. For a given $p \in \N$, a function $M \colon I^p \to I$ is a \emph{$p$-variable mean on $I$} if 
\Eq{*}{
\min(x) \le M(x) \le \max(x) \text{ for every }x \in I^p.
}
We can define some natural properties such as symmetry, continuity, convexity, etc. which refer to the property $M$ as a $p$-variable function. A mean $M$ is called \emph{strict} if $\min(x)< M(x) < \max (x)$ for every non-constant vector $x \in I^p$.

Now, define the order $\prec$ on the vector of real numbers of the same lengths by
\Eq{*}{
x\prec y :\iff \big(x_i \le y_i \text{ for every } i\big).
}
Then, a mean $M \colon I^p \to I$ is \emph{monotone} if $M(x)\le M(y)$ for all $x,y \in I^p$ with $x \prec y$.

A function $M \colon \bigcup_{p=1}^\infty I^p \to I$ is a \emph{mean} if all its $p$-variable restrictions $M\restr_p:=M|_{I^p}$ are means for all $p \ge 1$. Such a mean is called \emph{symmetric} (resp. \emph{continuous}, etc.) if all $M\restr_p$-s admit this property.

\subsection{Invariance in a family of means}
For $p \in \N$ a selfmapping $\MM \colon I^p \to I^p$ is called a \emph{mean-type mapping} if $\MM=(M_1,\dots,M_p)$ for some $p$-variable means $M_1,\dots,M_p$ on $I$. A mean $K \colon I^p \to I$ is called \emph{$\MM$-invariant} if $K \circ \MM=K$.

The most classical result by Borwein-Borwein \cite{BorBor87} states that if all means $M_i$ are continuous and strict, then there exists exactly one \mbox{$\MM$-invariant} mean.  This result has several generalisations (see, for example, Matkowski \cite{Mat09e}, and Matkowski-Pasteczka \cite{MatPas21}). For details we refer the reader to the reach literature on the subject, the classical one
Lagrange \cite{Lag84}, Gauss \cite{Gau18}, Foster-Philips \cite{FosPhi84a}, Lehmer \cite{Leh71}, Schoenberg \cite{Sch82} 
as well as more recent 
Baj\'ak--P\'ales \cite{BajPal09b,BajPal09a,BajPal10,BajPal13}, Dar\'oczy--P\'ales \cite{Dar05a,DarPal02c,DarPal03a}, Der\k{e}gowska--Pasteczka~\cite{DerPas2005},
G{\l}azowska \cite{Gla11b,Gla11a}, Jarczyk--Jarczyk \cite{JarJar18}, Matkowski \cite{Mat99b,Mat02b,Mat05,Mat09e}, Matkowski--P\'ales \cite{MatPal15}, Matkowski--Pasteczka \cite{MatPas20a,MatPas21} and Pasteczka \cite{Pas16a,Pas19a,Pas22b,Pas23b}.

In the next section, we will recall only the results from \cite{Pas23b}, as they are the most suitable for our purposes.

In this restricted (interval) setting, we know that a $\beta$-invariant extension is uniquely defined (in the family of means) if and only if the barycentric operator is power convergent (see \cite[Theorem~1]{MatPas21}). Moreover, the barycentric operator is a special case of a mean-type mapping consisting of extended means (see \cite{Pas23b}). Therefore, in what follows, we deliver a short introduction to these objects.

\section{Extended means}

 Observe that for $d, p\in \N$ and a $d$-variable mean $M \colon I^d \to I$ we can construct a $p$-variable mean by choosing $d$ indexes and applying the mean $M$ to the so-obtained vector of length $d$. Formally, for $d,p\in \N$ and a vector $\alpha \in \N_p^d$ we define a $p$-variable mean $M^{(p;\alpha)} \colon I^p \to I$ by
\Eq{E:parmean}{
M^{(p;\alpha)}(x_1,\dots,x_p):=M(x_{\alpha_1},\dots,x_{\alpha_d}) \text{ for all }(x_1,\dots,x_p)\in I^p.
}
Let us emphasize, that for $\alpha=(1,\dots,p)\in \N_p^p$, we have $M^{(p;\alpha)}=M$, thus (purely formally) each mean is an extended mean. However, this approach allowed us to establish several interesting results, which are dedicated to mean-type mappings consisting of such a means. We are going to recall them in the following section.

\subsection{Invariance of extended means}
Before we proceed to the invariance we need to build a mean-type mapping. The idea is to use an extended mean at each coordinate. Therefore, in some sense, we need to vectorise the previous approach. 

For $p \in \N$ and a vector $\dd=(d_1,\dots,d_p)\in \N^p$, let $\N_p^\mathbf{d}:=\N_p^{d_1}\times\dots\times \N_p^{d_p}$.
Using this notations, a sequence of means $\MM=(M_1,\dots,M_p)$ is called \emph{$\dd$-averaging} mapping on~$I$ if each $M_i$ is a $d_i$-variable mean on $I$. 

In the next step, for a $\dd$-averaging  mapping $\MM$ and a vector of indexes $\alpha =(\alpha_1,\dots,\alpha_p)\in \N_p^{d_1}\times\dots\times \N_p^{d_p}=\N_p^\dd$ let us define a mean-type mapping $\MM_\alpha \colon I^p \to I^p$ by 
\Eq{*}{
\MM_\alpha:=\Big(M_1^{(p;\alpha_1)},\dots,M_p^{(p;\alpha_p)}\Big),
}
where $M_i^{(p,\alpha_i)}$-s are defined by \eq{E:parmean}.

Observe that $\alpha$ is a $p$-tuple of sequences of elements in $\N_p$. Such an object can be represented as a directed graph. Therefore for a given $p\in\N$, $\dd=(d_1,\dots,d_p)\in \N^p$, and $\alpha \in \N_p^\dd$, we define the \emph{$\alpha$-incidence} graph $G_\alpha=(V_\alpha,E_\alpha)$ as follows: 
\Eq{*}{
    V_\alpha:=\N_p\quad\text{ and }\quad E_\alpha:=\big\{(\alpha_{i,j},i) \colon i \in \N_p \text{ and }j \in \N_{d_i}\big\}.
    }

Since the graph $G_\alpha$ plays a very important role in the invariance of $\MM_\alpha$, let us recall some elementary definitions from the graph theory.

A sequence $(v_0,\dots,v_n)$ of elements in $V$ such that $(v_{i-1},v_{i})\in E$ for all $i \in \{1,\dots,n\}$ is called a \emph{walk} from $v_0$ to $v_n$. The number $n$ is a \emph{length} of the walk. If for all $v, w \in V$ there exists a walk from $v$ to $w$, then $G$ is called \emph{irreducible}.
A \emph{cycle} is a non-empty walk in which only the first and last vertices are equal. A directed graph is said to be \emph{aperiodic} if there is no integer $k > 1$ that divides the length of every cycle of the graph. A~graph is called \emph{ergodic} if is simultaneously irreducible and aperiodic.

 Now we are ready to recall the main theorem from the paper \cite{Pas23b}. It turns out that the natural assumption to warranty that \mbox{$\MM_\alpha$-invariant} mean is uniquely determined, is that $G_\alpha$ is ergodic. 

\begin{prop}[\!\!\cite{Pas23b}, Theorem~2] \label{prop:oldmain}
Given an interval $I \subset \R$, parameters $p \in \N$, $\dd \in \N^p$, and a $\dd$-averaging mapping $\MM=(M_1,\dots,M_p)$ on $I$ such that all $M_i$-s are strict. For all $\alpha \in \N_p^\dd$ such that $G_\alpha$ is ergodic:
\begin{enumerate}[{\rm (a)}]
    \item \label{2.A} there exists the unique $\MM_\alpha$-invariant mean $K_\alpha \colon I^p \to I$; 
    \item \label{2.B} $K_\alpha$ is continuous;
    \item \label{2.C} $K_\alpha$ is strict;
    \item \label{2.D} $\MM_\alpha^n$ converges, uniformly on compact subsets of $I^p$, to the mean-type mapping $\KK_\alpha \colon I^p \to I^p$, $\KK_\alpha=(K_\alpha,\dots,K_\alpha)$;
    \item \label{2.E} $\KK_\alpha \colon I^p \to I^p$ is $\MM_\alpha$-invariant, that is $\KK_\alpha =\KK_\alpha\circ \MM_\alpha$;
    \item \label{2.F} if $M_1,\dots,M_p$ are nondecreasing with respect to each variable then so is $K_\alpha$;
    \item \label{2.G} if $I=(0,+\infty)$ and $M_1,\dots,M_p$ are  positively homogeneous, then every iterate of $\MM_\alpha$ and $K_\alpha$ are positively homogeneous.
\end{enumerate}
\end{prop}

\section{Results}
\subsection{Auxiliary results}
In what follows we show two results. The second one shall be considered a direct application of Proposition~\ref{prop:oldmain}. It is precessed by the simple (and purely technical) lemma, which shows that a key  assumption of this proposition is satisfied. 

\begin{lem}\label{lem:Qp}
For every $p \ge3$ the graph $Q_{p}:=\big(\N_p,\{(i,j) \in \N_p^2\colon i \ne j\}\big)$
is ergodic. 
\end{lem}

\begin{proof}
The irreducibility of $Q_p$ is trivial since there is an edge (in both directions) between each two distinct vertices. Furthermore, $(1,2)$ and $(1,2,3)$ are two cycles in $Q_p$ with a coprime length, which implies that $Q_p$ is aperiodic.
\end{proof}

Based on the above lemma, in view of Proposition~\ref{prop:oldmain}, part \eqref{2.A} the following proposition immediately follows.

\begin{prop}\label{prop:int}
Let $p \in \N$ with $p \ge 2$, and $M \colon I^p \to I$ be a symmetric, continuous, and strict mean.
Then there exists a unique $\beta_M$-invariant extension. Equivalently, the functional equation
\Eq{E:Ms}{
K(M(x^{\vee1}),\dots,M(x^{\vee(p+1)}))=K(x), \quad x \in I^{p+1}
}
has exactly one solution in the family of means $K \colon I^{p+1} \to I$.
\end{prop}
\begin{proof}
Let $\alpha\in \N_{p+1}^{p \times (p+1)}$ be given by
\Eq{Elem:alpha}{
\alpha=\big(\N_{p+1}\setminus\{1\}, \N_{p+1}\setminus\{2\},\dots,\N_{p+1}\setminus\{p+1\}\big).
}
Then $G_\alpha=Q_{p+1}$ and, by Lemma~\ref{lem:Qp}, it is ergodic. Thus, applying Proposition~\ref{prop:oldmain} part (a) to the $\dd$-averaging mapping $\MM$, where 
\Eq{Elem:dM}{
\dd:=(\underbrace{p,\dots,p}_{(p+1) \text{ times}}) \text{ and }
\MM:=(\underbrace{M,\dots,M}_{(p+1) \text{ times}}),
}
we get that there exists the unique $\MM_\alpha$-invariant mean.
In other words, the functional equation $K\circ \MM_\alpha=K$
has exactly one solution in the family of means $K \colon I^{p+1} \to I$.
However, in this setup, 
\Eq{*}{
\MM_\alpha(x)&=\big(M^{(p,\alpha_1)}(x),\dots,M^{(p,\alpha_{p+1})}(x)\big)\\
&=\Big( M\big((x_i)_{i \in \N_{p+1} \setminus\{1\}}\big),\dots, M\big((x_i)_{i \in \N_{p+1} \setminus\{p+1\}}\big)\Big)\\
&=(M(x^{\vee1}),\dots,M(x^{\vee(p+1)})),
}
thus the proof is completed.
\end{proof}

\subsection{Main result}
In view of Proposition~\ref{prop:int}, for every symmetric, continuous, and strict mean $M \colon I^p \to I$, we define its \emph{$\beta$-invariant extension} $\widetilde M \colon I^{p+1} \to I$ (denoted also as $M^\sim$) as the unique solution $K$ of equation \eq{E:Ms} in the family of means.

Now we show a few elementary properties of this extension. 
\begin{thm}\label{thm:Ipproperties}
Let $p \in \N$ with $p \ge 2$, and $M \colon I^p \to I$ be a symmetric, continuous, and strict mean. Then
\begin{enumerate}[(a)]
\item \label{thm:Ipproperties.cont} $\widetilde M$ is continuous;
\item \label{thm:Ipproperties.symm} $\widetilde M$ is symmetric;
\item \label{thm:Ipproperties.strict} $\widetilde M$ is strict;
\item \label{thm:Ipproperties.monot} if $M$ is monotone, then so is $\widetilde M$;
\item \label{thm:Ipproperties.conv} if $M$ is convex (concave) and monotone, then so is $\widetilde M$;
\item \label{thm:Ipproperties.homog} if $I=(0,+\infty)$ and $M$ is positively homogeneous, then so is $\widetilde M$.
\end{enumerate}
\end{thm}

\begin{proof} Set $\alpha$, $\dd$, and $\MM$ by equations \eq{Elem:alpha}, and \eq{Elem:dM}. Then $\widetilde M=K_\alpha$ is the unique $\MM_\alpha$-invariant mean. By Lemma~\ref{lem:Qp}, we know that $G_\alpha$ is ergodic.

 Therefore parts \eqref{thm:Ipproperties.cont}, \eqref{thm:Ipproperties.strict}, \eqref{thm:Ipproperties.monot}, \eqref{thm:Ipproperties.homog}, of this statement are implied by Proposition~\ref{prop:oldmain} parts \eqref{2.B}, \eqref{2.C}, \eqref{2.F},  \eqref{2.G}, respectively. Now, we need to prove only statements \eqref{thm:Ipproperties.symm} and \eqref{thm:Ipproperties.conv}.

To show \eqref{thm:Ipproperties.symm}, observe that for every vector $x \in I^{p+1}$ and a permutation $\sigma \colon \N_{p+1} \to \N_{p+1}$ we have
\Eq{*}{
\MM_\alpha(x \circ \sigma)
&=\big(M((x \circ \sigma)^{\vee 1}),\dots,M((x \circ \sigma)^{\vee(p+1)})\big)\\
&=\big(M(x^{\vee \sigma(1)}),\dots,M(x ^{\vee \sigma(p+1)})\big) \\
&=\big([\MM_\alpha(x)]_{\sigma(1)},\dots, [\MM_\alpha(x)]_{\sigma(p+1)} \big)
= \MM_\alpha(x) \circ \sigma.
}
Therefore, by Proposition~\ref{prop:oldmain} part \eqref{2.D}, for every $x \in I^{p+1}$ we have
\Eq{*}{
\KK_\alpha(x \circ \sigma)=\lim_{n \to \infty} \MM_\alpha^n(x \circ \sigma)=\big(\lim_{n \to \infty} \MM_\alpha^n(x) \big) \circ \sigma=\KK_\alpha(x) \circ \sigma,
}
thus, since $K_\alpha$ is (any of) entry in $\KK_\alpha$, we get $
K_\alpha(x \circ \sigma)=K_\alpha(x)$, which proves that $K_\alpha$ is symmetric.

Finally, to prove part \eqref{thm:Ipproperties.conv} assume that $M$ is convex and monotone and fix two vectors $x,y \in I^p$ and $t \in (0,1)$. As the intermediate step show that, for all $n \ge 0$,
\Eq{213}{
\big[\MM_\alpha^n(t x+(1-t)y)\big]_k \le t \big[\MM_\alpha^n(x)\big]_k+ (1-t) &\big[\MM_\alpha^n(y)\big]_k \\
&\text{ for all }k \in \N_{p+1}.
}
For $n=0$ inequality \eq{213} obviously holds (even with equality). Now assume that \eq{213} is valid for some $n \ge 0$. Then, for every $k_0 \in \N_{p+1}$, we have
\Eq{*}{
\big[\MM_\alpha^{n+1}(t x+(1-t)y)\big]_{k_0}
&=M\Big(\big(\MM_\alpha^{n}(t x+(1-t)y)\big)^{\vee k_0}\Big) \\
&\le M\Big(t\MM_\alpha^{n}(x)^{\vee k_0}+(1-t)\MM_\alpha^{n}(y)^{\vee k_0} \Big)\\
&\le t M\big(\MM_\alpha^{n}(x)^{\vee k_0}\big) + (1-t) M\big(\MM_\alpha^{n}(y)^{\vee k_0}\big)\\
&= t \big[\MM_\alpha^{n+1}(x)\big]_{k_0} + (1-t) \big[\MM_\alpha^{n+1}(y)\big]_{k_0}.
}
Since $k_0$ is an arbitrary element in $\N_{p+1}$ we obtain \eq{213} with $n$ replaced by $n+1$. Consequently \eq{213} is valid for all $n \ge 0$.  In the limit case as $n \to \infty$, in view of Proposition~\ref{prop:oldmain} part \eqref{2.D}, we get \Eq{*}{K_\alpha(t x+(1-t)y) \le K_\alpha(x)+(1-t)K_\alpha(y),}
which shows that $\widetilde M=K_\alpha$ is convex and, by the already proved part \eqref{thm:Ipproperties.monot}, monotone. The case when $M$ is concave, monotone is analogous.
 This completes the \eqref{thm:Ipproperties.conv} part of the proof.
 \end{proof}

In the following proposition, we show that this extension preserves the comparability of means.
\begin{prop}\label{prop:Ip-comp}
Let $p \in \N$ with $p \ge 2$, and $M, N \colon I^p \to I$ be symmetric, continuous, monotone, and strict means. Then $M \le N$ yields $\widetilde M \le \widetilde N$.
\end{prop}
\begin{proof} Let $\alpha$, $\dd$, and $\MM$ be like in equations \eq{Elem:alpha}, and \eq{Elem:dM}. 
Additionally, define $\mathbf{N}\colon (I^{p})^{p+1}\to I^{p+1}$ by $\mathbf{N}:=(N,\dots,N)$. However, $x^{\vee k} \prec y^{\vee k}$ for every $x,y \in I^{p+1}$ with $x \prec y$ and $k \in \N_{p+1}$. Therefore, since $M$ is monotone, we obtain 
\Eq{*}{
M(x^{\vee k}) \le M(y^{\vee k})\le N(y^{\vee k}) \text{ for every }k \in \N_{p+1},
}
which can be rewritten in a compact form as $\MM_\alpha(x) \prec \mathbf{N}_\alpha(y)$. Thus, by simple induction, $x \prec y$ implies  $\MM_\alpha^n(x) \prec \mathbf{N}_\alpha^n(y)$ for every $n \in \N$. In the limit case as $n \to \infty$, by Proposition~\ref{prop:oldmain} part \eq{2.D}, we obtain that $\widetilde M(x)\le \widetilde N(y)$ for every $x,y \in I^p$ with $x \prec y$. 

Finally, for every $x \in I^p$, we have $x \prec x$. Therefore, we obtain the desired inequality $\widetilde M(x) \le \widetilde N(x)$.
\end{proof}

\section{Multivariable generalization of bivariate means}
Let $M \colon I^2 \to I$ be a symmetric, continuous, and strict mean on an interval $I$. Then its \emph{iterative $\beta$-invariant extension} is a mean $M^e\colon \bigcup_{p=1}^\infty I^p \to I$ defined by
\Eq{E:MsR}{
M^e(x_1,\dots,x_p)=
\begin{cases}
x_1 &\text{ for }p=1;\\
M(x_1,x_2) &\text{ for }p=2;\\
M^{\sim(p-2)}(x_1,\dots,x_p) &\text{ for }p>2,
\end{cases}
}
where $M^{\sim p}$ is the $p$-th iteration of the $\beta$-invariant extension operator, that is
\Eq{*}{
M^{\sim1}=M^\sim=\widetilde M\text{ and }M^{\sim p}=(M^{\sim(p-1)})^\sim \text{ for }p \ge 2.
}

Now we adapt the result for the $\beta$-invariant extensions to the iterative setting.

\begin{prop}
Let $M \colon I^2 \to I$ be a symmetric, continuous, and strict mean. Then
\begin{enumerate}[(a)]
\item $M^e$ is continuous;
\item $M^e$ is symmetric;
\item $M^e$ is strict;
\item if $M$ is monotone, then so is $M^e$;
\item if $I=\R_+$ and $M$ is homogeneous, then so in $M^e$;
\item if $M$ is convex (concave) and monotone then so is $M^e$.
\end{enumerate}
\end{prop}

\begin{proof}
To show that $M^e$ is continuous recall that, in view of Theorem~\ref{thm:Ipproperties} part \eqref{thm:Ipproperties.cont}, the $\beta$-invariant extension operator preserves continuity. Therefore, once we know that $M^e\restr_p$ is continuous for certain $p\ge 2$, then so is $M^e\restr_{p+1}=(M^e\restr_p)^{\sim}$.  Thus, since $M^e\restr_2=M$ is continuous, by simple induction we find out that $M^e\restr_p$ is continuous for all $p\ge 2$ which is equivalent to the continuity of $M^e$ and completes the (a) part of the proof.

Analogously, applying other parts of Theorem~\ref{thm:Ipproperties}, we can show all the remaining parts of this statement.
\end{proof}

Next, we show that this operator is monotone with respect to the ordering of means.
\begin{prop}\label{prop:e-pres-ineq}
Let $M,\ N \colon I^2 \to I$ be symmetric, continuous, monotone, and strict means. Then $M \le N$ if and only if $M^e \le N^e$.
\end{prop}
\begin{proof}
First observe that, by Theorem~\ref{thm:Ipproperties} parts (a)--(c) and (f), all iterates $M^{\sim p}$ and $N^{\sim p}$ ($p \in \N$) are symmetric, continuous, monotone, and strict.

If $M\le N$ then, by Proposition~\ref{prop:Ip-comp}, one can (inductively) show that $M^{\sim p} \le N^{\sim p}$ for all $p \ge 1$, which is equivalent to the inequality $M^e\le N^e$.

Conversely, if $M^e \le N^e$ then $M=M^e\restr_2\le N^e\restr_2=N$.
\end{proof}

\subsection{Conjugacy of means}
For $n \in \N$, a mean $M \colon I^p \to I$ and a continuous, monotone function $\varphi \colon J \to I$
we define a \emph{conjugated mean} $M^{[\varphi]} \colon J^p \to J$ by
\Eq{*}{
M^{[\varphi]}(x_1,\dots,x_p):=\varphi^{-1} \circ M\big(\varphi(x_1),\dots,\varphi(x_p)\big).
}

Conjugacy of means is a generalization of the concept of quasiarithmetic means. As a matter of fact, one can easily check that quasiarithmetic means are simply conjugates of the arithmetic mean. In the next statement, we prove that iterative $\beta$-invariant extensions commute with conjugacy.

\begin{prop}
Let $I,J \subset \R$ be intervals and  $M \colon I^2\to I$ be a symmetric, continuous, monotone, strict mean, and $\varphi \colon J \to I$ be a continuous and strictly monotone function. Then $(M^e)^{[\varphi]}=(M^{[\varphi]})^e$.
\end{prop}
\begin{proof}
 First, we show that $\beta$-invariant extension commutes with conjugates. More precisely for every symmetric, continuous and monotone mean $N \colon I^p \to I$  we have \Eq{E:332}{
 (\widetilde N)^{[\varphi]}=(N^{[\varphi]})^\sim.
 }
 Indeed, define the mappings 
 \Eq{*}{
 {\bf M}\colon I^{p+1} \ni x&\mapsto \big(N(x^{\vee1}),\dots,N(x^{\vee(p+1)})\big) \in I^{p+1};\\
 {\bf P}\colon J^{p+1} \ni y&\mapsto \big(N^{[\varphi]}(y^{\vee1}),\dots,N^{[\varphi]}(y^{\vee(p+1)})\big) \in J^{p+1};\\
  \Phi\colon I^{p+1} \ni x&\mapsto\big(\varphi(x_1),\dots,\varphi(x_{p+1})\big) \in J^{p+1}.
 }
 Then, for all $y \in J^{p+1}$, we have
 \Eq{*}{
\Phi \circ {\bf P} (y)&=\Phi\big(N^{[\varphi]}(y^{\vee1}),\dots,N^{[\varphi]}(y^{\vee(p+1)})\big)\\
&=\big( N(\Phi (y)^{\vee1}),\dots N(\Phi (y)^{\vee(p+1)})\big)={\bf M}\circ \Phi (y).
}
Therefore ${\bf P} = \Phi^{-1} \circ {\bf M} \circ \Phi$, and whence 
\Eq{E:347}{
{\bf P}^n=\Phi^{-1} \circ {\bf M}^n \circ \Phi\text{ for all }n\ge 1.
}
By Proposition~\ref{prop:oldmain} part \eqref{2.D} we know that the sequence of iterates $({\bf M}^n)_{n=1}^\infty$ converge to $\widetilde N$ on each coordinate while $({\bf P}^n)_{n=1}^\infty$ converge to $(N^{[\varphi]})^{\sim}$ on each coordinate. Whence if we take the limit $n \to \infty$, equality \eq{E:347} implies $(N^{[\varphi]})^{\sim}=\varphi^{-1} \circ \widetilde N \circ \Phi=(\widetilde N)^{[\varphi]}$. Therefore \eq{E:332} holds.

By \eq{E:332} we easily obtain $ (M^{\sim p})^{[\varphi]}=(M^{[\varphi]})^{\sim p}$ for all $p \ge 1$, which implies
$(M^e)^{[\varphi]}=(M^{[\varphi]})^{e}$.
\end{proof}

\section{Applications to classical families of means}
\subsection{Quasiarithmetic envelopes}\label{sec:QA}
Quasi-arithmetic means were introduced as a generalization of Power Means in the 1920s/30s in a series of nearly simultaneous papers by de Finetti \cite{Def31}, Knopp \cite{Kno28},  Kolmogorov\cite{Kol30}, and Nagumo \cite{Nag30}. For an interval $I$ and a continuous and strictly monotone function $f \colon I \to \R$ (from now on $\CM(I)$ is a family of continuous, strictly monotone functions on $I$) we define \emph{quasiarithmetic mean} $\QA{f} \colon \bigcup_{n=1}^\infty I^n \to I$  by 
\Eq{*}{
\QA{f}(x_1,\dots,x_n):=f^{-1}\left( \frac{f(x_1)+f(x_2)+\cdots+f(x_n)}{n} \right),
}
where $n \in \N$ and $x_1,\dots,x_n \in I$. The function $f$ is called a \emph{generator} of the quasiarithmetic mean.

It is well known that for
$I=\R_+$, $\pi_r(t):=t^r$ for $r\ne 0$ and $\pi_0(t):=\ln t$, then mean $\QA{\pi_r}$ coincides with the $r$-th power mean (this fact had been already noticed by Knopp \cite{Kno28}).

There were a number of results related to quasi-arithmetic means. For example (see \cite{HarLitPol34}) $\QA{f}=\QA{g}$ if and only if their generators are affine transformations of each other, i.e. there exist $\alpha,\,\beta \in \R$ such that $g=\alpha f + \beta$.

Generalizing the approach from the paper \cite{Pas20a}, we are going to introduce quasiarithmetic envelopes. First, for a given mean $M \colon I^p\to I$ we define sets of quasiarithmetic means which are below and above $M$. More precisely let
\Eq{*}{
\sF^-(M)&:=\{f \in \CM(I) \colon \QA{f}(x)\le M(x)\text{ for all }x \in I^p\};\\
\sF^+(M)&:=\{f \in \CM(I) \colon \QA{f}(x)\ge M(x)\text{ for all }x \in I^p\}.
}

Now, for a given mean $M \colon I^p \to I$, we define \emph{local lower} and \emph{upper quasiarithmetic envelopes} $\LL_M,\UU_M \colon I^p \to I$ by
\Eq{*}{
\LL_M(x)&:=
\begin{cases}
\sup\{\QA{f}(x)\colon f \in \sF^-(M)\} & \text{ if }\sF^-(M)\ne \emptyset,\\
\min(x) & \text{ otherwise};
\end{cases}\\
\UU_M(x)&:=
\begin{cases}
\inf\{\QA{f}(x)\colon f \in \sF^+(M)\} & \text{ if }\sF^+(M)\ne \emptyset,\\
\max(x) & \text{ otherwise}.
\end{cases},
}
for all $x \in I^p$.

In the case $M \colon \bigcup_{p=1}^\infty I^p \to I$ we have two approaches to this problem. First, similarly to the previous case, we define two pairs $\LL_M,\UU_M \colon \bigcup_{p=1}^\infty I^p \to I$ by
\Eq{*}{
\LL_M(x_1,\dots,x_p)&:=\LL_{M\,\restr_p}(x_1,\dots,x_p), \\
\UU_M(x_1,\dots,x_p)&:=\UU_{M\,\restr_p}(x_1,\dots,x_p), 
}
where $p \in \N$ is arbitrary and $x_1,\dots,x_p \in I$.

On the other hand, we can repeat the previous setting from the beginning and define
\Eq{*}{
\sG^-(M)&:=\{f \in \CM(I) \colon \QA{f}(x)\le M(x)\text{ for all }x \in \bigcup_{p=1}^\infty I^p\};\\
\sG^+(M)&:=\{f \in \CM(I) \colon \QA{f}(x)\ge M(x)\text{ for all }x \in \bigcup_{p=1}^\infty I^p\}.
}

Then we set \emph{(global) lower} and \emph{upper quasiarithmetic envelopes} as follows. Let $\LLG_M,\UUG_M \colon \bigcup_{p=1}^\infty I^p \to I$ be given by
\Eq{*}{
\LLG_M(x)&:=
\begin{cases}
\sup\{\QA{f}(x)\colon f \in \sG^-(M)\} & \text{ if }\sG^-(M)\ne \emptyset,\\
\min(x) & \text{ otherwise};
\end{cases}\\
\UUG_M(x)&:=
\begin{cases}
\inf\{\QA{f}(x)\colon f \in \sG^+(M)\} & \text{ if }\sG^+(M)\ne \emptyset,\\
\max(x) & \text{ otherwise}.
\end{cases},
}
for all $x \in \bigcup_{p=1}^\infty I^p$.

Now we are going to show a few basic properties of operators $\LL$, $\UU$ and $\LLG$, $\UUG$. We bind them into two propositions. The first one, refers to $\LL$ and $\UU$, while the second to $\LLG$ and $\UUG$.

\begin{prop}\label{prop:PLU} The following properties are valid:
\begin{enumerate}[{\rm (i)}]
\item \label{PLU.1} For every mean $M \colon I^p \to I$ we have $\LL_M \le M \le \UU_M$.
\item \label{PLU.2} Let $M,N \colon I^p \to I$ be two means with $M \le N$. Then $\LL_M \le \LL_N$ and $\UU_M\le \UU_N$.
\item \label{PLU.3} For every $\varphi \in\CM(I)$ we have $\LL_{\QA{\varphi}}=\UU_{\QA{\varphi}}=\QA{\varphi}$.
\item \label{PLU.4} For every mean $M \colon I^p \to I$ and a monotone, continuous function $\varphi \colon J \to I$ we have $\LL_M^{[\varphi]}=\LL_{M^{[\varphi]}}$ and $\UU_M^{[\varphi]}=\UU_{M^{[\varphi]}}$.
\end{enumerate}
\end{prop}

\begin{proof}
We restrict the proof of this proposition to results concerning lower envelope. Parts concerning the upper envelope are analogous. 

To show \eq{PLU.1}, fix $x\in I^p$ be arbitrary. If $\sF^-(M)=\emptyset$ then 
\Eq{*}{
\LL_M(x)=\min(x)\le M(x).
}
Otherwise we have $\QA{f}(x) \le M(x)$ for all $f \in \sF^-(M)$ which yields 
\Eq{*}{
\LL_M(x)=\sup\{\QA{f}(x)\colon f \in \sF^-(M)\}  \le M(x),
}
which validates the inequality $\LL_M \le M$ and completes the proof of part~\eq{PLU.1}.

To prove part \eq{PLU.2}, it is sufficient to observe that, under the assumption $M\le N$, we have $\sF^-(M) \subseteq \sF^-(N)$, which easily implies $\LL_M\le \LL_N$.

Next, for every  $\varphi \in\CM(I)$ we have $\varphi \in \sF^-(\QA{\varphi})$, which shows the inequality $\LL_{\QA{\varphi}}\ge \QA{\varphi}$. Therefore $\LL_{\QA{\varphi}}= \QA{\varphi}$, since the second inequality has been already proved in \eq{PLU.1}. Thus we get \eq{PLU.3}.

Finally, in order to prove \eq{PLU.4} assume that $\varphi$ is strictly increasing and note that
\Eq{*}{
\sF^-(M^{[\varphi]})&=\{f \in \CM(J) \colon \QA{f}(x)\le M^{[\varphi]}(x)\text{ for all }x \in J^p\}\\
&=\{f \in \CM(J) \colon \QA{f\circ \varphi^{-1}}(x)\le M(x)\text{ for all }x \in I^p\}\\
&=\{f \in \CM(J) \colon f\circ \varphi^{-1} \in \sF^-(M)\}\\
&=\{g \circ \varphi \in \CM(J) \colon g \in \sF^-(M)\}.
}
Therefore either $\sF^-(M^{[\varphi]})=\sF^-(M)=\emptyset$ and the equality is trivial or for every $x=(x_1,\dots,x_p)\in I^p$ we have
\Eq{*}{
\LL_{M^{[\varphi]}}(x)
&=\sup\{\QA{f}(x)\colon f \in \sF^-(M^{[\varphi]})\}\\
&=\sup\{\QA{g \circ \varphi}(x)\colon g \in \sF^-(M)\}\\
&=\sup\big\{\varphi^{-1}\big(\QA{g}\big(\varphi(x_1),\dots,\varphi(x_p)\big)\big)\colon g \in \sF^-(M)\big\}\\
&=\varphi^{-1}\Big(\sup\big\{\QA{g}\big(\varphi(x_1),\dots,\varphi(x_p)\big)\colon g \in \sF^-(M)\big\}\Big)
=\LL_M^{[\varphi]}(x),
}
and we get \eq{PLU.4}, which was the last unproved part of this statement.
\end{proof}

In the same spirit, we can establish the analogous result for the global envelopes

\begin{prop}\label{prop:PLU-G} The following properties are valid:
\begin{enumerate}[{\rm (i)}]
\item \label{PLU-G.1} For every mean $M \colon \bigcup_{p=1}^\infty I^p \to I$ we have $\LLG_M \le M \le \UUG_M$.
\item \label{PLU-G.2} Let $M,N \colon \bigcup_{p=1}^\infty I^p \to I$ be a symmetric, continuous and strict means with $M \le N$. Then $\LLG_M \le \LLG_N$ and $\UUG_M\le \UUG_N$.
\item \label{PLU-G.3} For every $\varphi \in\CM(I)$ we have $\LLG_{\QA{\varphi}}=\UUG_{\QA{\varphi}}=\QA{\varphi}$.
\item \label{PLU-G.4} For every mean $M \colon I^p \to I$ and a monotone, continuous function $\varphi \colon J \to I$ we have $\LLG_M^{[\varphi]}=\LLG_{M^{[\varphi]}}$ and $\UUG_M^{[\varphi]}=\UUG_{M^{[\varphi]}}$.
\end{enumerate}
\end{prop}
Its proof follows the lines of proof of Proposition~\ref{prop:PLU}, but we need to replace the set $\sF^-(M)$ by $\sG^-(M)$.

\medskip 

In the next lemma, we prove that the $\beta$-invariant extension of a bivariate quasiaritmetic mean is the quasiarithmetic mean generated by the same function.
\begin{lem}
If $\varphi \in \CM(I)$ then $\big(\QA{\varphi}\restr_2\!\big)^e=\QA{\varphi}$.
\end{lem}
\begin{proof}
Let $\varphi \in \CM(I)$ and denote briefly $M:=\QA{\varphi}\restr_2$. We prove by induction that 
\Eq{E343}{
M^e\restr_p=\QA{\varphi}\restr_p
}
for all $p \ge 1$. For $p=1$ and $p=2$, this statement is trivial. 

Now assume that \eq{E343} holds for some $p=p_0\ge 2$. Then 
\Eq{*}{
M^e\restr_{p_0+1}=M^{\sim({p_0}-1)}=(M^{\sim({p_0}-2)})^{\sim}=(M^e\restr_{p_0})^{\sim}=(\QA{\varphi}\restr_{p_0})^{\sim}.
}
Consequently, for all $x \in I^{p_0+1}$, we have
\Eq{*}{
& \QA{\varphi}\big(\QA{\varphi}\big(x^{\vee1}\big),\dots,\QA{\varphi}\big(x^{\vee({p_0}+1)}\big)\big)\\
&\qquad\qquad=
\QA{\varphi}\bigg(\varphi^{-1} \Big(\frac{\varphi(x_1)+\dots+\varphi(x_{{p_0}+1})-\varphi(x_1)}{p_0}\Big),\dots\,\\
&\qquad\qquad\qquad\qquad ,\varphi^{-1} \Big(\frac{\varphi(x_1)+\dots+\varphi(x_{{p_0}+1})-\varphi(x_{{p_0}+1})}{p_0}\Big)\bigg)\\
&\qquad\qquad=\varphi^{-1} \bigg(\frac{1}{{p_0}+1} \sum_{k=1}^{p_0+1} \frac{\varphi(x_1)+\dots+\varphi(x_{p_0+1})-\varphi(x_k)}{p_0}\bigg)\\
&\qquad\qquad=\varphi^{-1} \bigg(\frac{1}{p_0+1} \sum_{k=1}^{p_0+1} \varphi(x_k)\bigg)=\QA{\varphi}(x).
}
Thus $M^e\restr_{p_0+1}=(\QA{\varphi}\restr_{p_0})^{\sim}=\QA{\varphi}\restr_{p_0+1}$, i.e. \eq{E343} holds for $p=p_0+1$. By the induction we obtain that \eq{E343} holds for all $p\ge 1$, which is precisely the equality $(\QA{\varphi}\restr_2)^e=M^e=\QA{\varphi}$. 
\end{proof}





Now we proceed to the comparability between envelopes. First, we show that local envelopes approximate the mean better than global ones, as expected.

\begin{prop}\label{prop:LLGcomp}
For every mean $M \colon \bigcup_{p=1}^\infty I^p \to I$ we have 
\Eq{*}{
\LLG_M\le \LL_M \le M \le \UU_M\le \UUG_M.
}
\end{prop}
\begin{proof}
Take $p \in \N$ and $x \in I^p$ arbitrarily. If $\sG^-(M)=\emptyset$ then $\LLG_M=\min \le \LL_M$. Otherwise, by $\sG^-(M)\subseteq \sF^-(M)$, we have
\Eq{*}{
\LLG_M(x)=\sup\{\QA{f}(x)\colon f \in \sG^-(M)\} \le \sup\{\QA{f}(x)\colon f \in \sF^-(M)\}=\LL_M(x),
}
which shows that 
\Eq{*}{
\LLG_M(x) \le \LL_M(x).
}

Similarly we obtain that either $\sF^-(M)=\emptyset$ and $\LL_M=\min$ or $\QA{f}(x) \le M(x)$ for all $f \in \sF^-(M)$, which yields
\Eq{*}{
\LL_M(x)=\sup\{\QA{f}(x)\colon f \in \sF^-(M)\}\le M(x),
}
which can be concluded as $\LL_M(x)\le M(x)$. Analogously we can prove the dual inequalities $M(x) \le \UU_M(x) \le \UUG_M(x)$.
\end{proof} 

In the next result, we show how the extension of a mean affects its envelopes.  

\begin{thm}
Let $M \colon I^2 \to I$ be a symmetric, continuous, monotone, and strict mean. Then 
\begin{enumerate}[(a)]
    \item $\sG^-(M^e)=\sF^-(M)$ and $\sG^+(M^e)=\sF^+(M)$;
    \item $\LLG_{M^e}=\sup\{\QA{f}\colon f \in \sF^-(M)\}$ and $\UUG_{M^e}=\inf\{\QA{f}\colon f \in \sF^+(M)\}$;
    \item $\LLG_{M^e}=\LL_{M}$ and $\UUG_{M^e}=\UU_{M}$ on $I^2$;
    \item $\LLG_{M^e}\le (\LL_M)^e \le M^e \le (U_M)^e \le \UUG_{M^e}$.
\end{enumerate}

\end{thm}
\begin{proof}
First, for all $f\in \CM(I)$, we have
\Eq{*}{
f \in \sG^-(M^e) &\iff \QA{f}(x)\le M^e(x)\text{ for all }x \in \bigcup_{p=1}^\infty I^p  \\
&\iff (\QA{f})^e(x)\le M^e(x)\text{ for all }x \in \bigcup_{p=1}^\infty I^p  \\
&\iff \QA{f}(x)\le M(x)\text{ for all }x \in I^2 \\
&\iff f \in \sF^-(M).
}
Thus $\sG^-(M^e)=\sF^-(M)$. The second equality is dual, thus we have proved (a). In view of this, we trivially obtain (b).

To show (c) note that, in the case $\sF^-(M) \ne \emptyset$, we have
\Eq{*}{
\LLG_{M^e}(x)=\sup\{\QA{f}(x)\colon f \in \sF^-(M)\}=\LL_{M}(x) \text{ for all }x \in I^2.
}
For $\sF^-(M) = \emptyset$ we obviously have  $\LLG_{M^e}=\min=\LL_M$ on $I^2$. Similarly $\UUG_{M^e}=\UU_M$ on $I^2$.

To prove (d), we focus on inequalities 
$\LLG_{M^e}\le (\LL_M)^e \le M^e$,
since the remaining ones are dual. Additionally one can assume that $\sF^-(M) \ne \emptyset$, since the second case is trivial.

First, note that 
\Eq{*}{
\LLG_{M^e}=\sup\{\QA{f}\colon f \in \sF^-(M)\}=\sup\{(\QA{f})^e\colon f \in \sF^-(M)\}.
}
However, for all $f \in \sF^-(M)$ we have
\Eq{*}{
(\QA{f})^e \le \big(\sup\{\QA{g}\colon g \in \sF^-(M)\}\big)^e= (L_M)^e,
}
thus $\LLG_{M^e}\le (L_M)^e$. Finally, Proposition~\ref{prop:e-pres-ineq} implies $(\LL_M)^e \le M^e$.
\end{proof}

\subsection{\label{sec:Gini} Gini means} 
Let us recall two results characterizing the comparison in the family of Gini means.

\begin{prop}[\!\!\cite{DarLos70}]
Let $p, q, r, s \in \R$. Then the following conditions are equivalent:
\begin{itemize}
\item $\G_{p,q}(x)\le \G_{r,s}(x)$ for all $n \in \N$ and $x \in \R^n$;
\item $\min(p,q) \le \min(r,s)$, and $\max(p, q) \le  \max(r , s)$;
\item $(p,q,r,s)\in \Delta_\infty$, where
\Eq{*}{
\Delta_\infty:=\{(p,q,r,s)\in \R^4\colon &\min(p,q) \le \min(r,s)\text{ and }\\
&\max(p, q) \le  \max(r , s)\}.
}
\end{itemize}
\end{prop}

\begin{prop}[\!\!\cite{Pal88c}, Theorem 3]
Let $p, q, r, s \in \R$. Then the following conditions are equivalent:
\begin{itemize}
\item For all $x, y > 0$, $\G_{p,q}(x,y)\le \G_{r,s}(x,y)$;
\item $p + q \le r + s$, $m(p, q) \le m(r , s)$, and $\mu(p, q) \le  \mu(r , s)$, where
\Eq{*}{
m(p,q):=
\begin{cases}
\min(p,q) & \text{ if }p,q\ge 0,\\
0& \text{ if } pq<0,\\
\max(p,q) & \text{ if }p,q\le 0;
\end{cases}
\quad 
\mu(p,q):=
\begin{cases}
\frac{\abs{p}-\abs{q}}{p-q} & \text{ if }p\ne q,\\
\sign(p) & \text{ if }p=q;
\end{cases}
}
\item $(p,q,r,s)\in \Delta_2$, where
\Eq{*}{
\Delta_2:=\{(p,q,r,s)\in \R^4&\colon p + q \le r + s,\ m(p, q) \le m(r , s),\\
&\qquad \mu(p, q) \le  \mu(r , s)\}.
}
\end{itemize}
\end{prop}

By \cite{Los71a,Los71b} we know that $\G_{p,q}$ is monotone if and only if 
\Eq{*}{
(p,q) \in \mathrm{Mon}_G:=\{(p,q)\in \R^2 \colon pq\le 0\}=m^{-1}(0).
}

As a straightforward application of Proposition~\ref{prop:e-pres-ineq} we get
\begin{cor}
Let $p, q, r, s \in \R$ with $(p,q), (r,s)\in \mathrm{Mon}_G$. 

Then $\G_{p,q}^e(x)\le \G_{r,s}^e(x)$ for all $n \in \N$ and $x \in \R_+^n$ if and only if $(p,q,r,s)\in \Delta_2$.
\end{cor}

This corollary shows that the equality $\G_{p,q}=\G_{p,q}^e$ fails to be valid for a large subclass of parameters $(p,q)$.


\end{document}